\documentclass[12pt]{article}
\usepackage{a4wide}
\usepackage{amsmath,amsfonts,amssymb,amsthm,amscd}
\usepackage{epsfig}

\usepackage{color}
\usepackage{enumerate}

\usepackage{hyperref}
\usepackage{breakurl}

\definecolor{modra3}{rgb}{.1,.0,.4}
\hypersetup{colorlinks=true, linkcolor=modra3, urlcolor=modra3, citecolor=modra3, pdfpagemode=UseNone, pdfstartview=} %colored links, no rectangles, no bookmarks, zoom

\theoremstyle{plain}

\newtheorem{theorem}{Theorem} %[section]

\newtheorem*{claim_A}{Claim A}
\newtheorem*{claim_B}{Claim B}

\begin{document}

%\title{Stronger Hanani--Tutte theorem}
\title{Unified Hanani--Tutte theorem}

\author{
Radoslav Fulek\thanks{IST, Klosterneuburg, Austria;
\texttt{radoslav.fulek@gmail.com}. Supported by the People Programme (Marie Curie Actions) of the European Union's Seventh Framework Programme (FP7/2007-2013) under REA grant agreement no [291734].}
  \and
Jan Kyn\v{c}l\thanks{Department of Applied Mathematics and Institute for Theoretical Computer Science, 
Charles University, Faculty of Mathematics and Physics, 
Malostransk\'e n\'am.~25, 118 00~ Praha 1, Czech Republic;
\texttt{kyncl@kam.mff.cuni.cz}. Supported by project 
16-01602Y of the Czech Science Foundation (GA\v{C}R).}
  \and
D\"om\"ot\"or P\'alv\"olgyi\thanks{University of Cambridge, UK;
\texttt{dom@cs.elte.hu}. Supported by the Marie Sk\l odowska-Curie
action of the EU, under grant IF 660400.} }

\date{}

\maketitle

\centerline{Dedicated to the 100th anniversary of William Thomas Tutte}

%==============================================================================================

\begin{abstract}
We introduce a common generalization of the strong Hanani--Tutte theorem and the weak Hanani--Tutte theorem: if a graph $G$ has a drawing $D$ in the plane where every pair of independent edges crosses an even number of times, then $G$ has a planar drawing preserving the rotation of each vertex whose incident edges cross each other evenly in $D$. The theorem is implicit in the proof of the strong Hanani--Tutte theorem by Pelsmajer, Schaefer and {\v{S}}tefankovi{\v{c}}. We give a new, somewhat simpler proof.
\end{abstract}

\section{Introduction}
\label{section_intro}
The Hanani--Tutte theorem~\cite{Ha34_uber,Tutte70_toward} is a classical result that provides an algebraic characterization of planarity with interesting theoretical and algorithmic consequences.
The (strong) Hanani--Tutte theorem states that a graph is planar if it can be drawn in the plane so that no pair of independent edges crosses an odd number of times.
Moreover, its variant known as the weak Hanani--Tutte theorem~\cite{CN00_thrackles,PT00_which,PSS06_removing} states that if $G$ has a drawing $\mathcal{D}$ where every pair of edges crosses an even number of times, then $G$ has an embedding that preserves the cyclic order of edges at each vertex of $\mathcal{D}$.
The weak variant earned its name because of its stronger assumptions; however, it does not directly follow from the strong variant since its conclusion is stronger than ``mere'' planarity.
For sub-cubic graphs, the weak variant implies the strong variant, since in this case pairs of adjacent edges crossing oddly can be dealt with by a local redrawing in a small neighborhood of each vertex.

We observe that there is a common generalization of both the strong and the weak variant, which seems to have been overlooked in the literature.

%\begin{theorem}[Stronger Hanani--Tutte theorem]
\begin{theorem}[Unified Hanani--Tutte theorem]
\label{theorem_stronger}
Let $G$ be a graph and let $W$ be a subset of vertices of $G$. Let $\mathcal{D}$ be a drawing of $G$ where every pair of edges that are independent or have a common endpoint in $W$ cross an even number of times. Then $G$ has a planar embedding where cyclic orders of edges at vertices from $W$ are the same as in $\mathcal{D}$.
\end{theorem}

By setting $W=\emptyset$ we obtain the strong Hanani--Tutte theorem, while $W=V(G)$ gives the weak variant.

Theorem~\ref{theorem_stronger} directly follows from the proof of the Hanani--Tutte theorem by Pelsmajer, Schaefer and {\v{S}}tefankovi{\v{c}}~\cite{PSS06_removing}. In Section~\ref{section_proof} we give a new, slightly simpler proof by induction, based on case distinction of the connectivity of $G$ and using the weak Hanani--Tutte theorem as a base case. Our proof of 
Theorem~\ref{theorem_stronger} also gives an alternative proof of the strong Hanani--Tutte theorem, by reducing it to the weak variant for $3$-connected graphs.

In Section~\ref{section_outro} we show how to extend Theorem~\ref{theorem_stronger} to multigraphs.

The unified Hanani--Tutte theorem can be used to simplify the proof of the Hanani--Tutte theorem for clustered planarity with two clusters~\cite{FKMP15_revisited}, and perhaps other variants of the Hanani--Tutte theorem as well. It is also one of the base cases of the induction in the proof of a variant for clustered planarity with embedded pipes~\cite{FK17_pipes}. 

%============================================================================================

\section{Notation}
We assume that $G=(V,E)$ is a graph, with no loops or multiple edges.
We use the shorthand notation $G-v$ for $G[V\setminus \{v\}]$. 
A \emph{drawing} of $G$ is a representation of $G$ in the plane where every vertex is represented by a unique point and every
edge $e=uv$ is represented by a simple curve joining the two points that represent $u$ and $v$. If it leads to no confusion, we do not distinguish between
a vertex or an edge and its representation in the drawing and we use the words ``vertex'' and ``edge'' in both contexts. We assume that in a drawing no edge passes through a vertex,
no two edges touch, every edge has only finitely many intersection points with other edges and no three edges cross at the same inner point. In particular, every common point of two edges is either their common endpoint or a crossing.

A drawing of a graph is an \emph{embedding} or a \emph{planar drawing} if no two edges cross.

The \emph{rotation} of a vertex $v$ in a drawing is the clockwise cyclic order of the edges incident to $v$. We will represent the rotation of $v$ by the cyclic order of the other endpoints of the edges incident to $v$. 

We say that two edges in a graph are \emph{independent} if they do not share a vertex.
An edge in a drawing is {\em even\/} if it crosses every other edge an even number of times.
A vertex $v$ in a drawing is {\em even\/} if all the edges incident to $v$ cross each other an even number of times.
A drawing of a graph is \emph{even} if all its edges are even.
A drawing of a graph is \emph{independently even} if every pair of independent edges in the drawing cross an even number of times.

%============================================================================================
\section{Proof of Theorem~\ref{theorem_stronger}}
\label{section_proof}

Assume that $G=(V,E)$ and let $n=|V|$.
We proceed by induction on $n$. The theorem is trivial for $n=1$. 

Assume that $n\ge 2$. We distinguish four cases according to the connectivity of $G$.

%- - - - - - - - - - - - - - - - - - - - - - - - - - - - - - - - - - - - - - - - - - - - - - -
\paragraph*{Case 0:} If $G$ is disconnected, we can use the inductive hypothesis for each component separately and draw the components in disjoint regions in the plane. 
%- - - - - - - - - - - - - - - - - - - - - - - - - - - - - - - - - - - - - - - - - - - - - - -
\paragraph*{Case 1:} $G$ is connected and has a separating vertex $v$. 
Let $G'_1, G'_2, \dots, G'_k$ be the connected components of $G-v$. For each $i\in[k]$, let $G_i=G[V(G'_i)\cup\{v\}]$ be the subgraph of $G$ induced by the vertices of $G'_i$ and $v$. By the inductive hypothesis, $G_i$ has an embedding $\mathcal{D}_i$ preserving the rotations of even vertices in $\mathcal{D}$. We may assume that $v$ is incident with the outer face of $\mathcal{D}_i$. By gluing all the drawings $\mathcal{D}_i$ at $v$ we obtain an embedding of $G$ preserving the rotations of all even vertices from $V(G)\setminus\{v\}$ in $\mathcal{D}$. If $v$ is not an even vertex in $\mathcal{D}$, the proof is finished.

If $v$ is an even vertex in $\mathcal{D}$, we need to glue the drawings $\mathcal{D}_i$ in a very particular way to preserve the rotation of $v$. To this end, we show that the rotation of $v$ in $\mathcal{D}$ is of a special form; see Figure~\ref{obr_obr_case_1}.

\begin{claim_A}%\label{claim_A}
If $v$ is even in $\mathcal{D}$, then there is an $i\in [k]$ such that the edges of $G_i$ incident with $v$ are consecutive in the rotation of $v$ in $\mathcal{D}$, and thus they form a well-defined clockwise linear order $\mathcal{R}_i$.
\end{claim_A}
%%%%old version
%If $v$ is even in $\mathcal{D}$, then for each $i\in [k]$, the edges of $G_i$ incident with $v$ are consecutive in the rotation of $v$ in $\mathcal{D}$, and thus they form a well-defined clockwise linear order $\mathcal{R}_i$.
%%%%

%no abab

\begin{proof}
Let $\mathcal{I}$ be a minimal cyclic interval of consecutive edges in the rotation of $v$ such that for some $i\in [k]$, all the edges of $G_i$ incident with $v$ are in $\mathcal{I}$. The claim will follow from the fact that $\mathcal{I}$ contains only edges of $G_i$.

Suppose that $\mathcal{I}$ contains an edge $e$ of $G_j$ for some $j\in[k]\setminus\{i\}$. From the minimality of $\mathcal{I}$, the edge $e$ is not first nor last in $\mathcal{I}$, and there is at least one more edge of $G_j$ incident with $v$ that does not belong to $\mathcal{I}$. That is, we have
distinct indices $i,j\in[k]$ and vertices $a,b\in V(G'_i)$, $c,d\in V(G'_j)$ such that the edges $va,vc,vb,vd$ appear in the rotation of $v$ in this cyclic order. Let $C_i$ be a cycle in $G_i$ extending the path $avb$ and let $C_j$ be a cycle in $G_j$ extending the path $cvd$. Since $\mathcal{D}$ is independently even and $v$ is even, every edge of $C_i$ crosses every edge of $C_j$ an even number of times in $\mathcal{D}$. However, the closed curves representing $C_i$ and $C_j$ in $\mathcal{D}$ cross at $v$, which implies that they have an odd number of common crossings; a contradiction.
\end{proof}

\begin{figure}
\begin{center}
\epsfbox{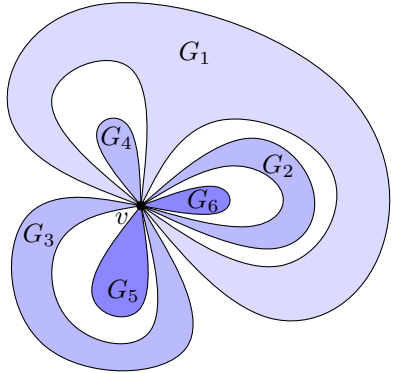}
\end{center}
\caption{A sketch of an embedding of a connected graph with a separating vertex $v$. Subgraphs $G_4$, $G_5$ and $G_6$ satisfy Claim A, while $G_1$, $G_2$ and $G_3$ do not.}
\label{obr_obr_case_1}
\end{figure}

Let $G_i$ be the subgraph from Claim A. By induction, we obtain an embedding $\mathcal{D}_i$ of $G_i$ and an embedding $\mathcal{E}_i$ of $G[V(G)\setminus V(G'_i)]$. We may choose the outer face of $\mathcal{D}_i$ so that when starting in the outer face, the clockwise linear order of the edges incident with $v$ in $\mathcal{D}_i$ is identical with the order $\mathcal{R}_i$. Now we can glue $\mathcal{D}_i$ in an appropriate face of $\mathcal{E}_i$ incident with $v$ to obtain an embedding of $G$ where the rotation of $v$ is the same as the rotation of $v$ in $\mathcal{D}$.

%- - - - - - - - - - - - - - - - - - - - - - - - - - - - - - - - - - - - - - - - - - - - - - -
\paragraph*{Case 2:} $G$ is $2$-connected and has a separating pair $u,v$. 
Let $G''_1, G''_2, \dots, G''_k$ be the connected components of $G[V\setminus \{u,v\}]$. For each $i\in[k]$, let $G'_i$ be the graph obtained from $G''_i$ by adding the vertices $u,v$ and all edges of $G$ joining $u$ or $v$ with $G''_i$. Let $G_i$ be the graph obtained from $G'_i$ by adding the edge $uv$. Since each $G'_j$ contains a path $P_j$ from $u$ to $v$, we can obtain an independently even drawing of each $G_i$ from the drawing of $G'_i$ and $P_j$, for some $j\neq i$, as follows. We start with the curve $\gamma_j$ representing $P_j$, and sequentially remove each self-crossing of $\gamma_j$ by a local redrawing; see Figure~\ref{obr_self_crossing}. In this way we obtain a simple curve joining $u$ with $v$, drawn in a small neighborhood of $\gamma_j$, and crossing every edge of $G'_i$ not incident with $u,v$ an even number of times. 

\begin{figure}
\begin{center}
\epsfbox{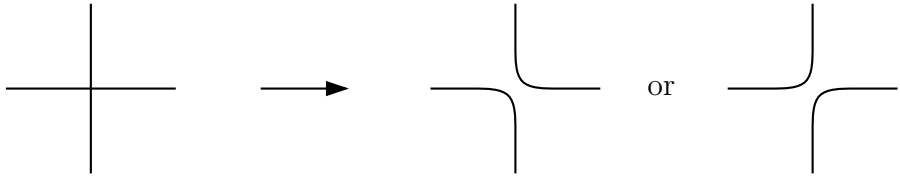}
\end{center}
\caption{Removing a self-crossing.}
\label{obr_self_crossing}
\end{figure}

By the inductive hypothesis, $G_i$ has an embedding $\mathcal{D}_i$ preserving the rotations of even vertices in $\mathcal{D}$. We may assume that the edge $uv$ is incident with the outer face of $\mathcal{D}_i$. By gluing all the drawings $\mathcal{D}_i$ at the edge $uv$ and possibly removing this edge we obtain an embedding of $G$ preserving the rotations of all even vertices of $G[V\setminus \{u,v\}]$ in $\mathcal{D}$. 

If $v$ is an even vertex, using Claim A for the graph $G-u$ we get that the edges of some subgraph $G'_i$ incident with $v$ are consecutive in the rotation of $v$ in $\mathcal{D}$, and thus they form a well-defined clockwise linear order. Moreover, this linear order coincides with the clockwise linear order of the edges incident with $v$ in $\mathcal{D}_i$, if we read the rotation of $v$ starting from the edge $uv$. In fact, since $G$ is $2$-connected, each subgraph $G'_i$ has this property in this case, as we show in Claim B. 

Let $H_i$ be a subgraph of $G$ obtained from $G[V(G)\setminus V(G''_i)]$ by adding the edge $uv$ if it is not present in $G$. We obtain an independently even drawing of $H_i$ by drawing the edge $uv$ along the curve $\gamma_i$ in $\mathcal{D}$ and by removing self-crossings.
By induction, we obtain an embedding $\mathcal{E}_i$ of $H_i$. We can now glue $\mathcal{D}_i$ in one of the two faces of $\mathcal{E}_i$ incident with $uv$ in such a way that the rotation of $v$ is preserved. Finally we remove the edge $uv$ if needed. If $u$ is not an even vertex, we are finished. The case when $u$ is even but $v$ is not is analogous.

%-!!-the added edge is important to prevent "folding" of components, should they have cut vertices of their own...!

If both $u$ and $v$ are even, we need the rotations of $u$ and $v$ in $\mathcal{D}$ to be ``compatible'' in the sense that the cyclic orders of the graphs $G'_i$ around $u$ and $v$ are opposite; see Figure~\ref{obr_obr_case_2}. If $uv\in E(G)$, we also have to include the graph $G'_0$ consisting of just the vertices $u,v$ and the edge $uv$, and we define $P_0$ as $G'_0$. The following claim finishes the proof of this case.

\begin{figure}
\begin{center}
\epsfbox{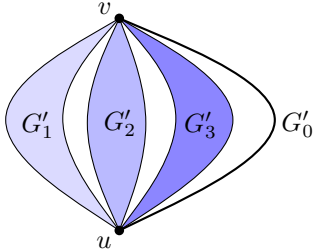}
\end{center}
\caption{A sketch of an embedding of a $2$-connected graph with a separating pair $u,v$.}
\label{obr_obr_case_2}
\end{figure}

\begin{claim_B}%\label{claim_B}
If $v$ is even in $\mathcal{D}$, then for each $i\in [k]$, the edges of $G'_i$ incident with $v$ are consecutive in the rotation of $v$ in $\mathcal{D}$, and thus they form a well-defined clockwise linear order. Moreover, this gives a well-defined clockwise cyclic order $\mathcal{C}_v$ of the graphs $G'_i$ around $v$. If both $u$ and $v$ are even, then the analogously defined order $\mathcal{C}_u$ is inverse to $\mathcal{C}_v$.
\end{claim_B}

\begin{proof}
Assume that $v$ is even. Suppose, for a contradiction, that there are indices $i,j,j'\in[k]\cup \{0\}$ with $i$ distinct from $j$ and $j'$, but $j$ and $j'$ possibly equal, and vertices $a,b\in V(G'_i)$, $c\in V(G'_j)$, $d\in V(G'_{j'})$ such that the edges $va,vc,vb,vd$ appear in the rotation of $v$ in this cyclic order.
Let $C_i$ be a cycle in $G'_i-u$ extending the path $avb$ and let $C_j$ be a cycle in $G[V(G)\setminus V(G''_i)]$ extending the path $cvd$. Notice that the cycles $C_i$ and $C_j$ share only the vertex $v$. Since $\mathcal{D}$ is independently even and $v$ is even, every edge of $C_i$ crosses every edge of $C_j$ an even number of times in $\mathcal{D}$. However, the closed curves representing $C_i$ and $C_j$ in $\mathcal{D}$ cross at $v$, which implies that they have an odd number of common crossings; a contradiction.

Now assume that both $u$ and $v$ are even. Let $i,j,l\in [k]\cup \{0\}$ be three distinct indices. The three paths $P_i,P_j,P_l$ form an independently even subdrawing of $\mathcal{D}$, which can be easily changed into an even drawing by a local change in the neighborhoods of the internal vertices of the three paths. Using the weak Hanani--Tutte theorem or the parity of the winding number of the cycle $P_i\cup P_j$ around inner points of the curve representing $P_l$, this implies that the cyclic orders of the paths around $u$ and $v$ are opposite. Consequently, the clockwise orders of the graphs $G'_i, G'_j, G'_l$ around $u$ and $v$ are opposite. Since this is true for all the triples $i,j,l$, the cyclic order $\mathcal{C}_u$ is inverse to $\mathcal{C}_v$.
\end{proof}

%- - - - - - - - - - - - - - - - - - - - - - - - - - - - - - - - - - - - - - - - - - - - - - -
\paragraph*{Case 3:} $G$ is $3$-connected. In this case we show that it is possible to change the rotations of the vertices locally to get an even drawing of $G$. The theorem will then follow from the weak Hanani--Tutte theorem.

Let $v$ be a vertex of $G$ with a pair of incident edges crossing oddly in $\mathcal{D}$. Let $uv$ be an arbitrary edge incident to $v$. By redrawing the other edges incident with $v$ in a small neighborhood of $v$, we can make them cross $uv$ an even number of times. Next, if two edges $f_1, f_2$ incident with $v$ and consecutive in the rotation of $v$ cross oddly, we can make them cross evenly by a local redrawing that swaps their position in the rotation, introduces exactly one crossing between $f_1$ and $f_2$, and does not change the parity of crossings of any other pair of edges. Let $\mathcal{D}'$ be a drawing of $G$ obtained after all these adjustments. Let $d$ be the degree of $v$ and let $(u_0=u,u_1,u_2,\dots,u_{d-1})$ be the rotation of $v$ in $\mathcal{D}'$. After the adjustments, $vu_0$ crosses every other edge $vu_i$ an even number of times, and for each $i\in \{1,2,\dots,d-2\}$, the edge $vu_i$ crosses the edge $vu_{i+1}$ an even number of times. We claim that this implies that $v$ is an even vertex in $\mathcal{D}'$; that is, every pair of edges $vu_i$, $vu_j$ crosses an even number of times.

Suppose, for a contradiction, that $v$ is not an even vertex in $\mathcal{D}'$. Let $i,j$ be indices with $i<j$ and with smallest difference $j-i$ such that $vu_i$ and $vu_j$ cross an odd number of times. Let $k=i+1$. Since $j-i\ge 2$, we have $i<k<j$. Among the edges $vu_0,vu_i,vu_k,vu_j$, only the pair $vu_i,vu_j$ crosses an odd number of times. Since $G$ is $3$-connected, the graph $G-v$ is $2$-connected. Hence, by Menger's theorem (see e.g. Diestel~\cite[Theorem 3.3.1]{Di16_fifth}), there are two vertex-disjoint paths between $\{u_0,u_k\}$ and $\{u_i,u_j\}$ in $G-v$. Together with the edges $vu_0,vu_i,vu_k,vu_j$, the paths form two edge-disjoint cycles $C_1,C_2$, which share only the vertex $v$. Except the pair $vu_i,vu_j$, every other pair of edges $e\in E(C_1)$ and $f\in E(C_2)$ crosses an even number of times. Moreover, the edges of the two cycles incident with $v$ do not alternate around $v$; that is, $C_1$ and $C_2$ ``touch'' at $v$. This implies that the two closed curves representing $C_1$ and $C_2$ in $\mathcal{D}'$ cross an odd number of times; a contradiction.

%============================================================================================
\section{Concluding remarks}\label{section_outro}
A reviewer raised the question whether a similar connectivity approach could also be used to prove the weak Hanani--Tutte theorem. We don't see how our approach could simplify currently known proofs, including the two recent beautiful redrawing proofs~\cite{PSS06_removing,FPSS12_adjacent}, which are quite simple already. 

A slightly different connectivity approach to the weak Hanani--Tutte theorem has been applied before: for example, there are weaker variants that hold for $2$-connected graphs. Schaefer~\cite{Sch13_hananitutte} surveyed several variants of planarity characterizations similar to the weak Hanani--Tutte theorem. One of them, due to Lov\'asz, Pach and Szegedy~\cite{LPS97_thrackle}, characterizes planar graphs as graphs that have a drawing where each $\Theta$-graph, that is, a subgraph that is a union of three paths with a common pair of endpoints $u,v$, is drawn so that the cyclic orders of the three paths around $u$ and $v$ are opposite. Schaefer~\cite[Theorem 1.9]{Sch13_hananitutte} showed that such a drawing can be changed to an embedding where the rotation system of every $2$-connected subgraph is preserved. The same conclusion is true for drawings where every cycle has an even number of self-crossings~\cite[Theorem 1.4]{Sch13_hananitutte}. Again, we do not see how our approach could bring new insights or simplifications to these variants.

%--------------------------------------------------------------------------------------------
\subsection{Multigraphs}
It is quite straightforward to generalize the strong and weak Hanani--Tutte theorems to multigraphs; in fact, Pelsmajer, Schaefer and {\v{S}}tefankovi{\v{c}}~\cite{PSS06_removing} proved the weak version directly for multigraphs, since they appeared naturally in their proof. The strong Hanani--Tutte theorem for multigraphs follows from the graph version using the fact that a multigraph is planar if and only if its underlying graph is planar.

Theorem~\ref{theorem_stronger} can also be extended to arbitrary multigraphs with multiple edges and loops, as follows.

\begin{theorem}[Unified Hanani--Tutte theorem for multigraphs]
%\label{theorem_stronger_multigraph}
Let $H$ be a multigraph and let $W$ be a subset of vertices of $H$. Let $\mathcal{D}$ be a drawing of $H$ where every pair of edges that are independent or have a common endpoint in $W$ cross an even number of times. Then $H$ has a planar embedding where cyclic orders of edges at vertices from $W$ are the same as in $\mathcal{D}$.
\end{theorem}

\begin{proof}
Given an independently even drawing $\mathcal{H}$ of a multigraph $H$, for each even vertex $v$ in $\mathcal{H}$ we subdivide each edge incident with $v$ by a new vertex placed very close to $v$. Each loop incident with $v$ is subdivided twice. The resulting drawing $\mathcal{H}'$ is still independently even, and no loop or multiple edge is incident to an even vertex of $\mathcal{H}'$. Let $\mathcal{H}''$ be a drawing of a graph $H''$ obtained from $\mathcal{H}'$ by removing all loops and multiple edges. By Theorem~\ref{theorem_stronger}, $H''$ has an embedding $\mathcal{E}$ preserving the rotation of all even vertices in $\mathcal{H}''$, which include all even vertices in $\mathcal{H}$. We add the removed edges and loops back to $\mathcal{E}$ without crossings, by drawing the multiple edges along their copies in $\mathcal{E}$, and drawing the loops close to the vertices they are incident to. Finally, by contracting the subdivided edges, we obtain an embedding of $H$ preserving the rotations of all even vertices in $\mathcal{H}$. 
\end{proof}

%============================================================================================
\section*{Acknowledgements}
We thank the reviewers for helpful suggestions, especially for noticing an error in a previous version of the proof.

%============================================================================================

%%%%%%%%%%%%%%%%%%%%%%%%%%%%%%%%%%%%%%%%%%%%%%%%%%%%%%%%%%%%%%%%%%%%%%%%%%%%%%%%%%%%%%%%%

\end{document}